\documentclass[10pt]{amsart}
\usepackage{amsfonts,amssymb,amscd,amsmath,enumerate,verbatim,calc,amsthm,color}
\headsep=1cm \numberwithin{equation}{section}
\newtheorem{thm}{Theorem}[section]
\newtheorem{cor}[thm]{Corollary}
\newtheorem{lem}[thm]{Lemma}
\newtheorem{prop}[thm]{Proposition}
\newtheorem{defn}[thm]{Definition}

\newtheorem{exam}[thm]{Example}
\newtheorem{rem}[thm]{Remark}

\newcommand{\Hom}{\mbox{Hom}\,}
\newcommand{\Ext}{\mbox{Ext}\,}
\newcommand{\Tor}{\mbox{Tor}\,}

\newcommand{\Ker}{\mbox{Ker}\,}
\newcommand{\Ass}{\mbox{Ass}\,}

\newcommand{\e}{\mbox{e}\,}
\newcommand{\Supp}{\mbox{Supp}\,}
\newcommand{\Sup}{\mbox{Sup}\,}

\newcommand{\Tot}{\mbox{Tot}\,}

\newcommand{\Log}{\mbox{Log}\,}

\newcommand{\gr}{\mbox{gr}\,}

\newcommand{\depth}{\mbox{depth}\,}
\renewcommand{\dim}{\mbox{dim}\,}
\renewcommand{\Im}{\mbox{Im}\,}

\newcommand{\K}{\mbox{K}\,}

\newcommand{\pd}{\mbox{proj.dim}\,}
\newcommand{\id}{\mbox{inj.dim}\,}
\newcommand{\fd}{\mbox{flatdim}\,}

\newcommand{\gd}{\mbox{G--dim}\,}

\newcommand{\h}{\mbox{ht}\,}

\newcommand{\im}{\mbox{Im}\,}
\newcommand{\E}{\mbox{E}}

\renewcommand{\H}{\mbox{H}}

\newcommand{\g}{\mbox{G}}

\newcommand{\fm}{\mathfrak{m}}
\newcommand{\fp}{\mathfrak{p}}
\newcommand{\fq}{\mathfrak{q}}
\newcommand{\fn}{\mathfrak{n}}

\begin{document}
\bibliographystyle{amsplain}


\title[sequence of exact zero--divisors]
 {sequence of exact zero--divisors}

\bibliographystyle{amsplain}

     \author[M. T. Dibaei]{Mohammad T. Dibaei$^{1}$}
     \author[M. Gheibi]{Mohsen Gheibi$^2$}

\address{ $^{1, 2}$ Faculty of Mathematical Sciences and Computer,
Tarbiat Moallem University, Tehran, Iran; and School of Mathematics,
Institute for Research in Fundamental Sciences (IPM), P.O. Box:
19395-5746, Tehran, Iran. }

\email{dibaeimt@ipm.ir}
 \email{mohsen.gheibi@gmail.com}

\keywords{exact zero--divisors, sequence of exact zero--divisors,
quasi-complete intersection ideal, complete intersection ring,
regular ring.\\
M.T. Dibaei was supported in part by a grant from IPM
(No.$\cdots$).}



\begin{abstract}
The concept of a sequence of exact zero--divisors on a
noetherian local ring is defined and studied. Some properties of
sequences of exact zero--divisors are compared with regular
sequences.
\end{abstract}

\maketitle

\bibliographystyle{amsplain}

\section{introduction}

Let $(R,\fm)$ be a noetherian local ring. According to I. B. Henriques and L. M. \c{S}ega \cite{HS},
an element $a\in \fm$ is said to be an exact zero-divisor if $(0:_Ra)$ is isomorphic to $R/(a)$. They studied the minimal
free resolution of a finitely generated $R$--module for which $\fm^4=0$ and admits an exact zero--divisor.
They also showed that $R$ is Gorenstein if and only if $R/(a)$ is Gorenstein ring(see \cite[Remark 1.6]{HS}).
Also as shown in \cite{AHS}, $\dim R=\dim R/(a)$ and $R$ is Cohen-Macaulay if and only if $R/(a)$ is Cohen-Macaulay
(see \cite[Theorem 3.3 and Proposition 4.3]{AHS}).
These properties shows that exact zero-divisors behave like regular elements. On the other hand when $a$
is an exact zero--divisor on $R$,
one has $\pd_R R/(a)=\infty$ while $\gd_R R/(a)=0$. Recall that a finitely generated $R$--module $M$ is said to
have G--dimension zero and write $\gd_R(M)=0$ if $M$ is reflexive
(i.e. the natural map $M\longrightarrow M^{**}$ is an isomorphism),
and $\Ext^i_R(M,R)=0= \Ext^i_R(M^*,R)$ for all $i>0$, where
$(-)^*:=\Hom_R(-,R)$.
In \cite{CJRSW} and \cite{H}, the authors construct infinitely many pairwise non-isomorphic indecomposable
$\g$--dimension zero modules by using exact zero--divisors.

In particular, in \cite[Theorem 1.4]{CJRSW}, it is shown constructively
that, if $R$ is a short local ring of length $e$ and algebraically
closed residue field such that it admits an exact zero-divisor, then for each
positive integer $n$, there exists an infinite family of
indecomposable pairwise non-isomorphic $\g$--dimension zero modules
of length $ne$ with periodic free resolutions of
period at most $2$.

Recently, the homological behaviors of modules over local rings
modulo exact zero--divisors have been studied and compared with the
homological behavior of modules over local rings modulo regular
elements (see \cite{BCJ}). These ideas have encouraged us to define
and study sequences of exact zero--divisors, discuss about their
existence and lengths in comparison with the regular sequences.

In section 2, we define sequence of exact zero--divisors on an
arbitrary module over a commutative noetherian ring and study basic
properties of rings and modules which admit a sequence of exact
zero--divisors.

In section 3, we study sequence of exact zero-divisors over artinian local rings. As main result of this section,
we show that a standard graded short local ring $R$ is Koszul complete intersection if and only if it admits
a minimal sequence of exact zero-divisors of length equal to socle degree of $R$.

In section 4, we define strong sequences of exact zero--divisors
and study some conditions, under which, over a noetherian local
ring, a sequence of exact zero--divisors is a strong sequence of
exact zero--divisors. Also we give an upper bound for the maximal
length of a strong sequence of exact zero--divisors on a local ring in
terms of the ring multiplicity. Finally, it is shown that, under
certain conditions, the quotient of a complete intersection ring by
a sequence of exact zero--divisors is a regular ring.


\section{sequence of exact zero--divisors}

Throughout $R$ is commutative noetherian ring with identity
element. The notion of an exact zero divisor is introduced in
\cite{HS}.

\begin{defn}\cite[Definition]{HS}\label{D1}\ \emph{Let $R$ be a local ring. A non-unit element $x\neq 0$ in $R$ is
said to be an} exact zero--divisor \emph{if one of the following
equivalent conditions holds.
\begin{itemize}
\item[(i)] $(0:_Rx)\cong R/(x)$.
\item[(ii)] There exists $y\in R$ such that the sequence
$$\cdots\longrightarrow R \overset{y}\longrightarrow
R\overset{x}\longrightarrow R \overset{y}\longrightarrow
R\overset{x}\longrightarrow R\longrightarrow0$$ is a free resolution of
$R/(x)$ over $R$.
\item[(iii)] There exists $y\in R$ such that
$(0:_Rx)=(y)$ and $(0:_Ry)=(x)$.
\end{itemize}
In this case $(x,y)$ is called a {\it pair} of exact zero--divisors.
Note that for fixed $x$, the element $y$ is unique up to
multiplication by a unit. We call $x$ (resp. $y$) the twin of $y$ (resp. $x$).}
\end{defn}

\begin{rem}\label{R1}\ \emph{It is easy to see that if $(x,y)$ is a pair of exact
zero--divisors, then $\gd_R(R/(x))=\gd_R(R/(y))=0$. Hence by
Auslander-Bridger formula $\depth R=\depth R/(x)=\depth R/(y)$. Also
by  \cite[Theorem 3.3]{AHS} $\dim R=\dim R/(x)=\dim R/(y)$. Thus $R$
is Cohen-Macaulay if and only if $R/(x)$ (equivalently, $R/(y)$) is (maximal)
Cohen-Macaulay. Moreover $R$ is Gorenstein if and only if $R/(x)$
(equivalently, $R/(y)$) is Gorenstein (see \cite[Remark 1.6]{HS}).}
\end{rem}
\begin{defn}\label{D5}\ \emph{Let $R$ be a (not necessarily local) ring and let $M$ be an
$R$--module (not necessarily finitely generated). Let $x$ be an
element of $R$. We say that $x$ is an} exact zero--divisor \emph{on
$M$ if $M\neq xM$, $x\notin 0:_RM$ and there is $y\in R$ such that
$0:_Mx=yM$ and $0:_My=xM$. In this case we call $(x,y)$ a}
\emph{pair of exact zero--divisors} \emph{on $M$. We refer to $y$ as
a {\it twin} of $x$. Let $x_1,\cdots, x_n$ be elements in $R$. We
call $x_1, \cdots, x_n$ a} sequence of exact zero--divisors \emph{on
$M$ if $x_i$ is an exact zero--divisor on $M/(x_1,\cdots,x_{i-1})M$
for all $i$, $1\leq i\leq n$. We call $x_1, \cdots, x_n$ is a
sequence of exact zero divisors on $R$ if it is so when $R$ is
considered  as an $R$--module.}
\end{defn}

\begin{prop}\label{P1}\ Let $R$ be a ring and let $M$ be an
$R$--module. Let $x_1, \cdots, x_n$ be a sequence of exact
zero--divisors on $M$. Then $0:_M(x_1, \cdots, x_n)\cong M/(x_1,
\cdots, x_n)M$. In particular, if $x_1, \cdots, x_n$ is a sequence
of exact zero--divisors on $R$, then $0:_R(x_1, \cdots, x_n)$ is a
principal ideal.
\end{prop}
\begin{proof}  Let $x\in R$ be an exact zero--divisor on $M$ so that there is $y\in R$ such that $0:_Mx=yM$ and $0:_My=xM$.
Clearly, the map $0:_Mx\longrightarrow M/xM $ defined by $ym\mapsto m+xM$ for all $m\in
M$ is an $R$--isomorphism.\\ Now let
$n>1$. Set $\overline{R}=R/(x_1)$ and $\overline{M}=M/x_1M$. Then
$\overline{x_2},\cdots,\overline{x_n}$ is a sequence of exact
zero--divisors on $M/x_1M$ over $R/(x_1)$. By induction we have
\[\begin{array}{rl}\
M/(x_1,\cdots, x_n)M&\cong0:_{\overline{M}}(\overline{x_2},\cdots,\overline{x_n})\\
&\cong\Hom_{R/(x_1)}(R/(x_1, \cdots, x_n),M/(x_1)M)\\
&\cong\Hom_{R/(x_1)}(R/(x_1, \cdots, x_n),\Hom_R(R/(x_1),M))\\
&\cong\Hom_R(R/(x_1, \cdots, x_n)\otimes_{R/(x_1)}R/(x_1),M)\\
&\cong\Hom_R(R/(x_1, \cdots, x_n),M)\\
&\cong 0:_M(x_1, \cdots, x_n).
\end{array}\]
\end{proof}

\begin{prop}\label{P10} Let $R$ be a local ring and let $M$ be a
finitely generated $R$--module. Assume that $x,y$ are elements of
$R$. Then the following statements are equivalent.
\begin{itemize}
\item[(i)] $(x,y)$ is a pair of exact zero--divisors on $M$.
\item[(ii)] $(x,y)$ is a pair of exact zero--divisors on $M/\alpha
M$ for any $M$--regular element $\alpha\in R$.
\item[(iii)] $(x,y)$ is a pair of exact zero--divisor
on $M/\alpha^nM$ for all $n\geq 1$ and for some $M$--regular element $\alpha\in R$.
\end{itemize}
\end{prop}
\begin{proof} (i)$\Rightarrow$(ii). Let $\alpha \in R$ be an $M$--regular element. As $yM=0:_Mx\cong M/xM$,
one has $\Ass_RM/xM\subseteq\Ass_RM$. Hence $\alpha$ is also an $M/xM$--regular element. The exact sequence
$$\cdots\overset{y}\longrightarrow M\overset{x}\longrightarrow M
\longrightarrow M/xM\longrightarrow 0,$$ by \cite[Proposition 1.1.5]{BH}, implies the exact sequence
$$\cdots\overset{y}\longrightarrow M/\alpha M
\overset{x}\longrightarrow M/\alpha M \longrightarrow
M/(x,\alpha)M\longrightarrow 0.$$ Thus $(x,y)$ is a
pair of exact zero--divisors on $M/\alpha M$.

(ii)$\Rightarrow$(iii) is clear.

(iii)$\Rightarrow$(i). We have
$xyM\subseteq\cap_{n\geq1}\alpha^nM=0$. Hence $yM\subseteq0:_Mx$.
Conversely, let $m\in 0:_Mx$. From the fact that  $(x,y)$ is a pair of exact zero--divisors
on $M/\alpha^nM$ for all $n\geq 1$, we have $m\in yM+\alpha^nM$ for
all $n\geq 1$. Therefore $m\in yM$. Similarly $0:_My=xM$.
\end{proof}

Now we are able to show that $M$ and $M/xM$ have equal depths whenever $x$ is an exact zero--divisor on $M$.
\begin{cor}\label{C10} Assume that $(R,\fm,k)$ is a local ring, $M$ a
finitely generated $R$--module and that $(x,y)$ is a pair of exact
zero--divisors on $M$. Then $$\emph\depth_RM=\emph\depth_RM/xM=\emph\depth_RM/yM.$$
\end{cor}
\begin{proof} Note that if $\alpha_1,\cdots,\alpha_n$ is an
$M$--regular sequence then it follows from Proposition \ref{P10}, that $\alpha_1,\cdots,\alpha_n$
is also an $M/xM$ and $M/yM$--regular
sequence. Hence $\depth_RM/xM\geq\depth_RM$. As, by Proposition \ref{P10}, $x$ is an exact zero--divisor on
$M/(\alpha_1,\cdots,\alpha_n)M$,
we may assume that $\depth_RM=0$. Note that, by Proposition \ref{P1}, $\Hom_R(R/(x),M)\cong M/xM$. Then we have
\[\begin{array}{rl} \Hom_R(k,M/xM)&\cong\Hom_R(k,\Hom_R(R/(x),M))\\
&\cong\Hom_R(k\otimes_RR/(x),M)\\&\cong\Hom_R(k,M)\not =0.\\
\end{array}\]
Thus $\depth_RM/xM=0$. For $M/yM$ we treat in the same way.
\end{proof}

In the following result we prove that, over local ring $R$, if
$(x,y)$ is a pair of exact zero--divisors on a finitely generated
$R$--module $M$, then $M$, $M/xM$ and $M/yM$ have equal dimensions.
Its proof is similar to that of \cite[Theorem 3.3]{AHS}. Here we
bring the proof for convenience of the reader.

First we need the following lemma which is the module version of
\cite[Theorem 4.1]{FH}. Note that $\ell(-)$ denotes the length
function.

\begin{lem}\label{N} Let $R$ be a local ring and let $M$ be a finitely generated
$R$--module with $\dim M=1$. Let $x\in R$ be a parameter element of
$M$ and set $y=ux$, $u\in R$. If the map $M/xM
\overset{u}\longrightarrow M/yM$ given by $m+xM\longmapsto um+yM$ is
injective, then $y$ is also a parameter element of $M$.
\end{lem}
\begin{proof} We may assume that $\dim R=1$. As $x$ is a parameter element of $M$, $0:_Mx$ has finite length.
Hence it follows from the exact sequence
$$0\longrightarrow 0:_M(x,u)\longrightarrow 0:_Mx \overset{u}\longrightarrow 0:_Mx\longrightarrow
\frac{0:_Mx}{u(0:_Mx)}\longrightarrow 0$$
that $\ell(\frac{0:_Mx}{u(0:_Mx)})=\ell(0:_M(x,u))$.

Assume contrarily that $y$ is not a parameter element of $M$. Hence
$u$ is not a parameter element of $M$. Then there exists a prime
ideal $\fp\in\min\Supp M$ such that $u\in \fp$, so that
$\dim_R(0:_Mu)>0$. Set $N=0:_Mu$. We have $e(x,N)>0$, where $e(x,N)$
is the multiplicity of $N$ with respect to the parameter ideal
$(x)$. By \cite[Theorem 4.7.4]{BH}, we have
$\ell(N/xN)-\ell(0:_Nx)=e(x,N)>0$ and as $0:_Nx=0:_M(x,u)$, we
obtain
$$\ell(N/xN)>\ell(\frac{0:_Mx}{u(0:_Mx)}).$$ Now as the map
$M/xM\overset{u}\longrightarrow M/yM$ is injective, we have
$N\subseteq xM$ and therefore $N=x(0:_Mxu)$. Hence
$$N/xN=\frac{x(0:_Mxu)}{xN}\cong \frac{0:_Mxu}{N+(0:_Mx)}\cong
\frac{u(0:_Mxu)}{u(0:_Mx)}\subseteq \frac{0:_Mx}{u(0:_Mx)}.$$ It
follows that $\ell(\frac{0:_Mx}{u(0:_Mx)})\geq \ell(N/xN)$, which is
a contradiction.
\end{proof}

\begin{prop}\label{Pr} Let $(R,\fm,k)$ be a local ring and let $M$ be a
finitely generated $R$--module of dimension $d$. Let $(x,y)$ be a
pair of exact zero--divisors on $M$. Then
$\emph\dim_RM=\emph\dim_RM/xM=\emph\dim_RM/yM$.
\end{prop}
\begin{proof} For $d=0$ we have nothing to prove.

We first consider the case $\dim_RM=1$. Assume contrarily that
$\dim_RM/xM=0$. Hence $x$ is a parameter element of $M$. As $(x,y)$
is a pair of exact zero-divisors on $M$, the map
$M/xM\longrightarrow M$ given by $m+xM \mapsto ym$ is injective. It
follows from Lemma \ref{N} that $xy$ is a parameter element of $M$
which is a contradiction because $xyM=0$.

Now let $d>1$ and assume  contrarily that $\dim_RM/xM<d$. Choose
$\fq\in\Supp_RM$ such that $\dim R/\fq=d$. Hence $x\notin\fq$ and
 $\dim_RM/xM= \dim R/(\fq+Rx)$.
There is a prime ideal $\fp$ such that $\fq+Rx\subseteq \fp$ and
$\dim_R M/xM= \dim R/\fp=d-1$ and so $\h\fp/\fq=1$. As $(x,y)$ is a
pair of exact zero--divisors on $M$, by the exact sequence
$$0\longrightarrow M/xM\longrightarrow M\longrightarrow M/yM\longrightarrow 0$$ we obtain $(M/yM)_{\fq}\cong M_{\fq}\neq 0$.
 Hence $y\in\fq\subset \fp$.
On the other hand, $(x/1,y/1)$ is a pair of
exact zero--divisors on $M_{\fp}$ with $\dim_{R_\fp}(M_\fp)=1$. As $\dim_{R_\fp}(M_\fp/xM_\fp)= 0$, by the previous case this is a
contradiction.
\end{proof}


\begin{cor}\label{C9} Let $R$ be a local ring and $M$ be a
finitely generated $R$--module. Let $x_1, \cdots, x_n$ be a sequence of exact zero--divisors on
$M$. Then $M$ is Cohen-Macaulay $R$--module if and only if $M/(x_1,\cdots,x_n)M$ is Cohen-Macaulay $R$--module.
\end{cor}
\begin{proof} By induction it follows from Corollary \ref{C10} and Proposition \ref{Pr}.
\end{proof}


\begin{lem}\label{L3} Let $R$ be a ring, $(x,y)$ a pair of exact zero--divisors on $R$. Assume that $M$ is an $R$--module
and set $\overline{R}=R/(x)$. Consider the following statements.
\begin{itemize}
\item[(i)] $(x,y)$ is a pair of exact zero--divisors on $M$.
\item[(ii)] $\emph\Ext^i_R(\overline{R},M)=0$ for all $i>0$.
\item[(iii)]  $\emph\Tor^R_i(\overline{R},M)=0$ for all $i>0$.
\end{itemize}
Then \emph{(i)}$\Rightarrow$\emph{(ii)}$\Leftrightarrow$\emph{iii}. If $M\not=xM$ and $x\notin(0:_RM)$ then
the statements \emph{(i)},\emph{(ii)} and \emph{(iii)} are equivalent.
\end{lem}
\begin{proof}It follows from the exact sequence
$$\cdots\overset{x}\longrightarrow R\overset{y}\longrightarrow
R\overset{x}\longrightarrow R\longrightarrow\overline{R}\longrightarrow 0$$
that (ii) and (iii) are equivalent. The second part is clear.
\end{proof}

\begin{exam}\label{E1} \emph{If $(x,y)$ is a pair of exact zero--divisors on $R$, it
follows from Lemma \ref{L3} that $(x,y)$ is a pair of exact
zero--divisors on any injective,
projective or flat $R$--module $M$, whenever $M\neq xM$ and $x\notin 0:_RM$.}
\end{exam}


\begin{cor}\label{P4}\ Let $R$ be a Cohen-Macaulay local ring of
dimension $d$ with canonical module $\omega_R$. Let $x_1, \cdots,
x_n$ be a sequence of exact zero--divisors on $R$. Then $x_1, \cdots,
x_n$ is a sequence of exact zero--divisors on $\omega_R$. More
precisely, $R/(x_1, \cdots, x_n)$ is Cohen-Macaulay of dimension $d$
with the canonical module $\omega_R/(x_1, \cdots, x_n)\omega_R$.
\end{cor}
\begin{proof}\ It is enough to prove the case $n=1$. Let $x$ be an exact zero--divisor on $R$.
Since $R/xR$ is maximal Cohen-Macaulay $R$--module by Remark \ref{R1}, we have
$\Ext^i_R(R/xR,\omega_R)=0$ for all $i>0$. Hence by Lemma \ref{L3},
$x$ is an exact zero--divisor on $\omega_R$. On the other hand,
$0:_{\omega_R}x\cong\omega_R/x\omega_R$ by Proposition \ref{P1}. As $0:_{\omega_R}x$
is the canonical module of $R/xR$, we are done.
\end{proof}


\begin{prop}\label{T4}\ Let $R$ be a ring and let $M$ be an
$R$--module. Let $x_1, \cdots, x_n$ be a sequence of exact
zero--divisors on both $R$ and $M$ with the same twins
$y_1,\cdots,y_n$, respectively. Set $\overline{R}=R/(x_1, \cdots,
x_n)$, $\overline{M}=M/(x_1, \cdots, x_n)M$ and let $N$ be an
$\overline{R}$--module. Then the following statements hold true.
\begin{itemize}
\item[(i)] $\emph\Ext^i_R(N,M)\cong\emph\Ext^i_{\overline{R}}(N,\overline{M})$ for all $i\geq
0$.
\item[(ii)] $\emph\Ext^i_R(M,N)\cong\emph\Ext^i_{\overline{R}}(\overline{M},N)$ for all $i\geq
0$.
\item[(iii)]  $\emph\Tor^R_i(M,N)\cong\emph\Tor^{\overline{R}}_i(\overline{M},N)$ for all $i\geq 0$.
\end{itemize}
\end{prop}
\begin{proof}\ We prove the case $n=1$ and for $n>1$ the result follows inductively.

(i) Let $0\longrightarrow M\longrightarrow I^\bullet$ be an injective
resolution of $M$. Since $(x,y)$ is a pair of exact zero--divisor on $R$ and on
$M$, by Lemma \ref{L3},
$0\longrightarrow\Hom_R(\overline{R},M)\longrightarrow\Hom_R(\overline{R},
I^\bullet)$ is exact and it is an injective resolution of $\Hom_R(\overline{R},M)\cong\overline{M}$ as
$\overline{R}$--module. Hence
$\Ext^i_{\overline{R}}(N,\overline{M})\cong\H^i(\Hom_{\overline{R}}(N,\Hom_R(
\overline{R},I^\bullet)))\cong
\H^i(\Hom_R(N\otimes_{\overline{R}}\overline{R},I^\bullet))\cong\H^i(\Hom_R
(N,I^\bullet))=\Ext^i_R(N,M)$.

(ii) By Lemma \ref{L3}, $\Tor^R_i(M,\overline{R})=0$ for all $i>0$.
Let $F_{\bullet}\longrightarrow M\longrightarrow 0$ be a free
resolution of $M$. Then
$F_{\bullet}\otimes_R\overline{R}\longrightarrow
M\otimes_R\overline{R}\longrightarrow 0$ is exact so that it is a
free resolution of $\overline{M}$ as $\overline{R}$--module. Hence
$\Ext^i_R(M,N)=\H^i(\Hom_R(F_{\bullet},N))\cong\H^i(\Hom_R(F_{\bullet},
\Hom_{\overline{R}}(\overline{R},N)))\cong
\H^i(\Hom_{\overline{R}}(F_{\bullet}\otimes_R\overline{R},N))=\Ext^i_{
\overline{R}}(\overline{M},N)$.

(iii) Similarly,
\[\begin{array}{rl}
\Tor^R_i(M,N)&= \H_i(F_{\bullet}\otimes_RN)\\
&\cong\H_i(F_{\bullet}\otimes_R(\overline{R}\otimes_{\overline{R}}
N))\\
&\cong\H_i((F_{\bullet}\otimes_R\overline{R})\otimes_{\overline{R}}N)\\&=
\Tor^{\overline{R}}_i(\overline{M},N).
\end{array}\]
\end{proof}


\begin{cor}\label{C4}\ Let $(R,\fm,k)$ be a local ring and let $M$ be a
finitely generated $R$--module. Assume that $x_1, \cdots, x_n$ is a sequence
of exact zero--divisors on both $R$ and $M$ with the same twins
$y_1,\cdots,y_n$, respectively. Then
$\emph\pd_R(M)=\emph\pd_{\overline{R}}(\overline{M})$ and
$\emph\id_R(M)=\emph\id_{\overline{R}}(\overline{M})$.
\end{cor}
\begin{proof}\ By Proposition \ref{T4}, we have $\pd_R(M)=\Sup\{i\mid\Tor^R_i(M,k)\neq 0\}=
\Sup\{i\mid\Tor^{\overline{R}}_i(\overline{M},k)\neq
0\}=\pd_{\overline{R}}(\overline{M})$ and
$\id_R(M)=\Sup\{i\mid\Ext^i_R(k,M)\neq
0\}=\Sup\{i\mid\Ext^i_{\overline{R}}(k,\overline{M})\neq
0\}=\id_{\overline{R}}(\overline{M})$.
\end{proof}


\begin{cor}\label{L5}\ Let $f:R\longrightarrow S$ be a homomorphism of
rings such that $S$ is flat $R$--module via $f$. Let $x_1, \cdots,
x_n$ be a sequence of exact zero--divisors on $R$. If $f(x_1),
\cdots, f(x_n)$ are non-zero and non-unit elements of $S$, then $f(x_1), \cdots, f(x_n)$ is a
sequence of exact zero--divisors on $S$. Moreover $\overline{f}:\overline{R}\longrightarrow \overline{S}$
is flat ring homomorphism, where $\overline{(-)}=(-)/(x_1,\cdots,x_n)\overline{(-)}$.
\end{cor}
\begin{proof}\ It is enough to prove for the case $n=1$. We have $\Tor^R_i(R/(x_1),S)=0$
for all $i>0$. As $x_1$ is an exact zero--divisor on $R$, by Lemma
\ref{L3}, $x_1$ is an exact zero--divisor on $R$--module $S$ and so
that $f(x_1)$ is an exact zero--divisor on $S$. Note that by Proposition \ref{T4}, we have
$\Tor^{\overline{R}}_i(N,\overline{S})\cong \Tor^R_i(N,S)=0$ for all $i>0$ and all $\overline{R}$--module $N$.
Hence $\overline{S}$ is flat as $\overline{R}$--module.
\end{proof}


\begin{rem}\label{R2}\ \emph{It follows inductively
from Lemma \ref{L3} and Proposition \ref{T4}, that if $x_1, \cdots,
x_n$ is a sequence of exact zero--divisors on both $R$ and $M$ with
the same twins $y_1,\cdots,y_n$, respectively, then
$\Ext^i_R(\overline{R},M)=\Tor^R_i(M,\overline{R})=0$  for all
$i>0$, where $\overline{R}=R/(x_1, \cdots, x_n)$. Moreover one easily finds that $\gd_R\overline{R}=0$.}
\end{rem}


\begin{lem}\label{L4}\ Let $R$ be a ring and let $(x,y)$ be a pair
of exact zero--divisors on $R$. Assume that $0\longrightarrow
M_1\longrightarrow M_2\longrightarrow M_3\longrightarrow 0$ is an
exact sequence of $R$--modules such that $M_j\not=xM_j$ and $x\notin
0:_RM_j$ for all $ j= 1, 2, 3$. Then if $(x,y)$ is a pair of exact
zero--divisors on two out of three modules $M_1$, $M_2$ and $M_3$,
then it is a pair of exact zero--divisors on the third one. More
precisely the sequence $0\longrightarrow
\overline{M_1}\longrightarrow \overline{M_2}\longrightarrow
\overline{M_3}\longrightarrow 0$ of $\overline{R}$--modules is
exact, where $\overline{(-)}=-\otimes_R R/(x)$.
\end{lem}
\begin{proof}\ Assume that $(x,y)$ is a pair of exact zero--divisors on $M_1$ and $M_2$.
Applying the functor $\Hom_R(\overline{R},-)$ to the exact sequence
$0\longrightarrow M_1\longrightarrow M_2\longrightarrow
M_3\longrightarrow 0$ gives $\Ext^i_R(\overline{R},M_3)=0$ for all
$i>0$, by Lemma \ref{L3}. Hence the result follows from Lemma
\ref{L3}. If $(x,y)$ is a pair exact zero--divisors on $M_2$ and $M_3$, then
applying the functor $\overline{(-)}$ we get $\Tor^R_i(M_1,\overline{R})=0$ for all $i>0$. Again, by Lemma
\ref{L3}, the result follows. The remaining part is treated in the same way.
\end{proof}

Let $R$ be a local ring and let $M$ be a finitely generated $R$--module. Let $\mathbf{F}\longrightarrow M$
be a minimal free resolution of $M$. Assume that $(x,y)$ is a pair of exact zero-divisors on $R$. It follows from Lemma \ref{L4} that if
$(x,y)$ is a pair of exact zero-divisors on at least one syzygy module of $M$, then it is a pair of exact zero-divisors on
$M$ and on all the syzygy modules of $M$. Moreover in this case $\mathbf{F}/x\mathbf{F}\longrightarrow M/xM$
is a minimal free resolution
of $M/xM$ as $R/(x)$--module.


\begin{prop}\label{P7} Let $R$ be a ring and $(x,y)$ be a pair of exact zero--divisors on $R$. Let $M$ be an
$R$--module such that $M\not=xM$ and $x\notin(0:_R M)$ and set $\overline{(-)}=-\otimes_R R/(x)$.
Assume that at least one of $\id_R M$, $\pd_R M$, or $\fd_R M$ is finite. Then the following statements hold true.
\begin{itemize}
\item[(a)] $(x, y)$ is a pair of exact zero--divisors on $M$.
\item[(b)] $\emph\id_{\overline{R}}\overline{M}\leq \emph\id_RM$,
 $\emph \pd_{\overline{R}}\overline{M}\leq\emph\pd_RM$, or
$\emph\fd_{\overline{R}}\overline{M}\leq\emph\fd_RM$.
\item[(c)] If $R$ is local and $M$ is
finitely generated $R$--module, then each inequality in \emph{(b)}, if exists, is equality.
\end{itemize}
\end{prop}
\begin{proof} (a) and (b). We just prove in the case where $\id_R(M)<\infty$ and for the others the proof is similar.
First assume that $\id_R(M)=0$. Then Example \ref{E1} implies
that $(x,y)$ is a pair of exact zero--divisors on $M$. On the other hand,
Proposition \ref{T4} implies that $\Ext^i_{\overline{R}}(N,
\overline{M})=\Ext_R^i(N, M)= 0$ for all $\overline{R}$--module $N$
and all $i>0$. Thus $\overline{M}$ is injective
$\overline{R}$--module.

Let $\id_R(M)=1$. Consider an exact sequence
\begin{equation}\label{D1} 0\longrightarrow
M\longrightarrow I\longrightarrow M'\longrightarrow 0
\end{equation} such that $I$ and $M'$ are injective $R$--modules.
Hence we have
$\Ext^{i>0}_R(\overline{R},I)=0=\Ext^{i>0}_R(\overline{R},M')$.
 By Lemma \ref{L3},
$\Tor^R_{i>0}(\overline{R},I)=0=\Tor^R_{i>0}(\overline{R},M')$.
Hence $\Tor^R_{i>0}(\overline{R},M)=0$.  Again Lemma
\ref{L3} implies that $(x,y)$ is a pair of exact zero--divisors on
$M$. Now applying $\overline{(-)}$ on (\ref{D1}) implies the exact sequence
\begin{equation}\label{D2} 0\longrightarrow
\overline{M}\longrightarrow\overline{I}\longrightarrow\overline{M^{\prime}}
\longrightarrow 0.
 \end{equation} From (\ref{D1}) and (\ref{D2}) it follows that $(x, y)$ is also a pair of exact zero divisors on $I$ and so $\overline{I}$ is injective $\overline{R}$--module. Note that if $\overline{M'}\neq0$, then $\Hom_R(\overline{R}, M')\cong \overline{M'}$ so that $\overline{M'}$ is injective as $\overline{R}$--module. Thus $\id_{\overline{R}}\overline{M}\leq 1$. Now we proceed by induction.

If $R$ is local and $M$ is finitely generated, then the equality
holds by Corollary \ref{C4} and the fact that finitely generated flat modules coincide with projective modules.
\end{proof}

\begin{prop}\label{P12} Let $R$ be a local ring and let $M$ be a finitely generated $R$--module.
Assume $(x,y)$ is a pair of exact zero-divisors on both $R$ and $M$
and that $\emph\Ext^i_R(M,R)=0$ for all $i>0$. Then we have
$$\emph\Ext^n_R(M/xM,R)\cong\emph\Tor^R_n(R/(x),M^{*})$$ for all
$n\geq 0$, where $M^{*}=\emph\Hom_R(M,R)$.

In particular if $M$ is $\emph\g$--dimension zero $R$--module, then
$M/xM$ is $\emph\g$--dimension zero $R$--module\emph{(}$R/(x)$--module\emph{)}, if and only if $(x,y)$
also is a pair of exact zero-divisor on $M^{*}$.
\end{prop}
\begin{proof} Choose $F_{\bullet}:\cdots\longrightarrow R\overset{y}\longrightarrow R\overset{x}\longrightarrow
R\longrightarrow R/(x)\longrightarrow 0$,
 the minimal free resolution of $R$--module $R/(x)$, and let
$I^{\bullet}:0\longrightarrow R\longrightarrow I^{0}\longrightarrow
I^{1}\longrightarrow\cdots$ be an injective resolution of $R$.
Consider the third quadrant double complex
$X:\Hom_R(\Hom_R(F_{\bullet},M),I^{\bullet})$. Let $^{{\i}}\E$
(resp. $^{{\i}{\i}}\E$) denote the spectral sequence determined by
the first filtration(resp. second filtration) of $\Tot(X)$. Then we
have $^{{\i}}\E^{i,j}_2\cong \Ext^j_R(\Ext^i_R(R/(x),M),R)$. As
$(x,y)$ is a pair of exact zero-divisors on $M$, Lemma \ref{L3}
implies that $\Ext^i_R(R/(x),M)=0$ for all $i>0$. Hence by
Proposition \ref{P1},
$$^{{\i}}\E^{i,j}_2\cong \left\lbrace
           \begin{array}{c l}
              \Ext^j_R(M/xM,R)\ \ & \text{ \ \ $i=0$,}\\
              0\ \   & \text{   \ \ $\textrm{otherwise}$.}
           \end{array}
        \right.$$\\
Next for computing $^{{\i}{\i}}\E_2$, we use the functorial
isomorphism $\Hom_R(\Hom_R(F_{\bullet},M),I^{\bullet})\cong
F_{\bullet} \otimes_R\Hom_R(M,I^{\bullet})$ of double complexes.
Hence we have $^{{\i}{\i}}\E^{i,j}_2\cong
\Tor^R_i(R/(x),\Ext^j_R(M,R))$. As $\Ext^j_R(M,R)=0$ for all $j>0$, we have
$$^{{\i}{\i}}\E^{i,j}_2\cong \left\lbrace
           \begin{array}{c l}
              \Tor^R_i(R/(x),M^{*})\ \ & \text{ \ \ $j=0$,}\\
              0\ \   & \text{   \ \ $\textrm{otherwise}$.}
           \end{array}
        \right.$$\\
As the two spectral sequences $^{{\i}}\E$ and $^{{\i}{\i}}\E$
collapse, we have $^{{\i}}\E_{\infty} = {^{{\i}}\E_2}$ and
$^{{\i}{\i}}\E_{\infty} = {^{{\i}{\i}}\E_2}$. Thus one obtain the
isomorphisms $\Ext^n_R(M/xM,R)\cong\Tor^R_n(R/(x),M^{*})$ for all
$n\geq 0$.

For the next part, the fact that $\gd_R(M/xM)=0 $ implies that
$\Tor^R_i(R/(x),M^{*})=0$ for all $i>0$. Hence, by Lemma \ref{L3},
$(x,y)$ is a pair of exact zero-divisors on $M^{*}$. Conversely, let
$(x,y)$ be a pair of exact zero-divisors on $M^{*}$. By Lemma
\ref{L3}, $\Tor^R_n(R/(x),M^{*})=0$ for all $n>0$. Hence
$\Ext^n_R(M/xM,R)=0$ for all $n>0$. On the other hand, replacing $M$
with $M^{*}$, we have $\Ext^n_R(M^{*}/xM^{*},R)\cong
\Tor^R_n(R/(x),M^{**})\cong \Tor^R_i(R/(x),M)=0$ for all $n>0$. By
the first part, we have
$M^{*}/xM^{*}\cong\Hom_R(M/xM,R)=(M/xM)^{*}$. So we have
$\Ext^n_R((M/xM)^{*},R)=0$ for all $n>0$. Now it remains to show
that $M/xM$ is reflexive $R$--module. As $M$ is reflexive, the
natural map $M\longrightarrow M^{**}$ is an isomorphism. Hence
tensoring by $R/(x)$ gives the natural isomorphism $R/(x)\otimes_R
M\longrightarrow R/(x)\otimes_R M^{**}$ which implies
$M/xM\cong(M/xM)^{**}$.

Note that as by Proposition \ref{T4}, $\Ext^n_{R/(x)}(M/xM,R/(x))\cong\Ext^n_R(M/xM,R)$, with the same argument one can see
that $M/xM$ also is a $\g$--dimension zero $R/(x)$--module if and only if $(x,y)$ is a pair of exact zero-divisors
on $M^*$.
\end{proof}




\section{sequence of exact zero--divisors and complete intersection rings}

Throughout this section $R$ is a noetherian ring. We denote the number of a minimal generator of the finitely generated
$R$--module $M$ by $\mu_R(M)$. A sequence of elements $x_1, \cdots, x_n$ of $R$ is called {\it minimal} if $\mu(x_1, \cdots, x_n)=n$.

Let $R$ be a local ring and $I$ be an ideal of $R$ with a generating
set$\{a_1, \cdots, a_t\}$. Let $F$ be a free $R$--module of rank $t$
with the standard basis $\{e_1, \cdots, e_t\}$ and let
$f:F\longrightarrow R$ be an $R$--linear map such that $f(e_i)=a_i$.
Let $\K_{\bullet}(I)$ be the Koszul complex of $R$ with respect to
$a_1, \cdots, a_t$ and denote its homology modules by $\H_{*}(I)$.
In particular $\H_0(I)=R/I$. By \cite[Proposition 1.6.2]{BH},
$\K_{\bullet}(I)$ carries the structure of associated graded
$R$--algebra namely $\bigwedge^R F$ where $\bigwedge^R F$ is the
exterior algebra of the $R$--module $F$. By \cite[Proposition
1.6.4]{BH}, $\H_{*}(I)$ has a structure of graded associated
$R/I$--algebra inherited from $\K_{\bullet}(I)$ and there is a
unique homomorphism
$\lambda^{R/I}_*:\bigwedge^{R/I}_*\H_1(I)\longrightarrow\H_*(I)$ of
graded $R/I$--algebras with $\lambda^{R/I}_1$ is identity(see
\cite[Sections 1.6 and 2.3]{BH} for more details).

Now recall from
\cite{AHS} the notion of quasi-complete intersection ideal.
\begin{defn}\ \label{d2}\emph{An ideal $I$ of a local ring $R$ is called} quasi-complete intersection \emph{if
$\H_1(I)$ is a free $R/I$--module and the canonical homomorphism
of graded $R/I$--algebras
\centerline{$\lambda^{R/I}_*:\bigwedge^{R/I}_*{\H_1(I)}\longrightarrow\H_*(I)$} is
bijective.}
\end{defn}

As mentioned in \cite[1.6]{AHS}, a principal ideal is a
quasi-complete intersection if and only if it is generated by an
$R$--regular element or an exact zero--divisor. It follows by \cite[7.8]{AHS} and induction
that any sequence of exact zero-divisors is a quasi-complete intersection ideal. Also one can
conclude directly from the following result that any sequence of exact zero-divisors
is a quasi-complete intersection ideal.


\begin{prop}\label{T2} Let $R$ be a local ring and let $x_1, \cdots, x_n$
be non-zero and non-unit elements of $R$. Assume that
$\emph{K}_\bullet(x_1, \cdots, x_n)$ is the Koszul complex of $R$
with respect to the ideal generated by $x_1, \cdots, x_n$ and that
$\emph{H}_i(x_1, \cdots, x_n)$ denotes $i$th homology. Then the
following statements are equivalent.
\begin{itemize}
\item[(i)] $x_1, \cdots, x_n$ is a sequence of exact zero--divisors.
\item[(ii)] $x_p\notin(x_1,\cdots,x_{p-1})$ and $\emph{H}_i(x_1, \cdots, x_p)$ is a free
$R/(x_1, \cdots, x_p)$--module of rank $\large{(}^p_i\large{)}$
for all $p$, $1\leq p\leq n$, and all $i$, $0\leq i\leq p$.
\end{itemize}
\end{prop}
\begin{proof} For each $p$ and $i$ with $0\leq p\leq n$ and $0\leq i\leq p$, by \cite[Theorem 16.4]{M},
there is the long exact sequence
$$\cdots\overset{(-1)^ix_p}\longrightarrow\H_i(x_1, \cdots, x_{p-1})\longrightarrow\H_i(x_1, \cdots, x_p)\longrightarrow
\H_{i-1}(x_1, \cdots,
x_{p-1})\overset{(-1)^{i-1}x_p}\longrightarrow\cdots,$$ which gives
the short exact sequence
\begin{equation}\label{d1} 0\longrightarrow\H_i(x_1, \cdots, x_{p-1})/x_p
\H_i(x_1,\cdots, x_{p-1})\longrightarrow\H_i(x_1, \cdots, x_p)\longrightarrow
0:_{\H_{i-1}(x_1, \cdots, x_{p-1})}x_p\longrightarrow 0.
\end{equation}

(i)$\Rightarrow$(ii). We proceed by induction on $n$. For $n=1$ the
claim is clear. Let $n>1$ and $1\leq p\leq n$.  By induction
$\H_j(x_1, \cdots, x_{p-1})$ is a free $R/(x_1, \cdots,
x_{p-1})$--module of rank $(^{p-1}_j)$ for all $j=0,1, \cdots, p-1$.
Hence $\H_i(x_1, \cdots, x_{p-1})/x_p\H_i(x_1, \cdots, x_{p-1})$ is
a free $R/(x_1, \cdots, x_p)$--module of rank $(^{p-1}_i)$. On the
other hand, $0:_{\H_{i-1}(x_1, \cdots, x_{p-1})}x_p$ is a free
$R/(x_1, \cdots, x_p)$--module of rank $(^{p-1}_{i-1})$. Therefore
the short exact sequence \ref{d1} splits and so $\H_i(x_1, \cdots,
x_p)$ is a free $R/(x_1, \cdots, x_p)$--module of rank
$(^{p-1}_i)+(^{p-1}_{i-1})=(^p_i)$.

 (ii)$\Rightarrow$(i). For $n=1$ it is
clear. Let $n>1$ and $1\leq p\leq n $. Set $i=p$ in (\ref{d1}), we
have $\H_p(x_1, \cdots, x_p)\cong 0:_{ \H_{p-1}(x_1, \cdots,
x_{p-1})}x_p$. By assumptions, the right hand side is isomorphic to
$R/(x_1, \cdots, x_p)$, left hand side is isomorphic to $0:_{R/(x_1,
\cdots, x_{p-1})}x_p$ and $x_p\notin(x_1,\cdots,x_{p-1})$. Therefore
$x_p$ is an exact zero--divisor on $R/(x_1,\cdots,x_{p-1})$.
\end{proof}


\begin{prop}\label{T5} Let $(R,\fm)$ be a local ring and let $x_1, \cdots, x_n$ be a set of minimal
generators of $\fm$. Then the following statements are equivalent.
\begin{itemize}
\item[(i)] $x_1, \cdots, x_n$ is a sequence of exact zero--divisors.
\item[(ii)] $R/(x_1, \cdots, x_i)$ is an artinian complete intersection ring for
each $i$, $0\leq i\leq n$.
\end{itemize}
\end{prop}
\begin{proof}\ (i)$\Rightarrow$(ii). As $\fm$ is a quasi--complete intersection ideal so by \cite[Theorem
2.3.11]{BH}, $R$ is complete intersection. As mentioned in Remark \ref{R1}, $\dim R=\dim R/(x_1, \cdots, x_n)=0$ and so $R$ is
artinian. The same argument holds true for the sequence $\overline{x_{i+1}}, \cdots, \overline{x_n}$ in the ring
$\overline{R}=R/(x_1, \cdots, x_i)$ and so $R/(x_1, \cdots, x_i)$ is a complete
intersection artinian ring for each $i$, $1\leq i\leq n$.\\

(ii)$\Rightarrow$(i). It follows inductively by \cite[7.5 and 7.8]{AHS}.
\end{proof}

\begin{cor}\label{C21} Let $(R,\fm)$ be an artinian local ring such that $\fm^{n+1}=0$ and $\fm^n\neq 0$.
If $R$ admits a sequence of exact zero-divisors $x_1, \cdots, x_n$ such that $\mu_R(x_1, \cdots, x_n)=n$, then $R$
is a complete intersection.
Moreover, assume that $x_1,\cdots,x_n$ is such a sequence with the twins $y_1,\cdots,y_n$,
respectively. Then $\fm=(x_1, \cdots, x_n)$ and $y_j\notin \fm^2$ for all $j$, $1\leq j\leq n$.
\end{cor}
\begin{proof} First note that as $\fm^{n+1}=0$ we have $\fm^ny_1=0$. Hence $\fm^n\subseteq Rx_1$.
It follows that in the ring $R/(x_1)$, $n$th power of the maximal ideal is zero. Proceeding in this way we see
that, over $R/(x_1,\cdots,x_{n-1})$, square power of the maximal ideal is zero. But as $x_n$ is an exact zero-divisor
on $R/(x_1,\cdots,x_{n-1})$, one can see that $\fm/(x_1,\cdots,x_{n-1})=(x_1,\cdots,x_n)/(x_1,\cdots,x_{n-1})$.
Hence $\fm=(x_1,\cdots,x_n)$ and by Proposition \ref{T5}, $R$ is a complete intersection.

Note that if $y_1\in \fm^2$ then we will have $\fm^{n-1}\subseteq Rx_1$ and as seen in the last part,
we will have $\fm=(x_1,\cdots,x_{n-1})$, which is a contradiction. By the same way we can conclude
that $y_j\notin \fm^2$ for all $j$,  $1\leq j\leq n$.
\end{proof}


\begin{exam} \emph{Let $R=K[X]/(X^4)$ where $K$ is a field. Denote by $x$ the image of $X$ in $R$. Then $x^3,x^2,x$
is a sequence of exact zero-divisors on $R$ but it is not minimal.}
\end{exam}


 Let $(R,\fm,k)$ be a local ring and $M$ be a finitely generated
$R$--module. Denote by $M^g=\gr_R(M)$, the associated graded module of $M$ with respect to the maximal
ideal $\fm$. Note that if $F$ is free $R$--module, then $F^g$ is free $R^g$--module.
Let $\mathbf{F}\longrightarrow M$ be a minimal free resolution of $M$. Recall, from \cite{HI}, that $M$ is called a
 {\it Koszul} $R$--{\it module}
if the induced complex
$$\mathbf{F}^g:\cdots\longrightarrow F_n^g(-n)\longrightarrow F_{n-1}^g(-n+1)\longrightarrow \cdots \longrightarrow
F_0^g\longrightarrow 0$$ is acyclic. In other words $M$ is a Koszul $R$--module if and only if
$M^g$ has a linear resolution over $R^g$. The ring $R$ is called a {\it Koszul ring} if its residue field $k$ is Koszul
$R$--module.
See \cite{HI} and \cite[\S2]{S} for details.

\begin{prop}\label{P15} Let $(R,\fm,k)$ be an artinian local ring such that $\fm^{n+1}=0$ and
$\fm^n\neq 0$. Then the following statements hold true.
\begin{itemize}
\item[(i)] If $n=2$, then $R$ is Koszul complete intersection if and only if $R$ admits
a minimal sequence of exact zero-divisors of length $2$.
\item[(ii)] If $R$ is a standard graded $k$--algebra and admits a minimal sequence of exact zero-divisors
of length $n$, then $R$ is a Koszul complete intersection.
\item[(iii)] If $n=3$ and $R$ is a standard graded $k$--algebra with $k$ algebraically closed field of characteristic zero,
then $R$ is Koszul complete intersection if and only if $R$ admits
a minimal sequence of exact zero-divisors of length $3$.
\end{itemize}
\end{prop}
\begin{proof} (i) Assume $R$ is a Koszul complete intersection. Then by \cite[Corollary 4.5]{HS},
$\H_R(t)=(1+t)^2$, where $\H_R(t)$ is the Hilbert series of $R$. It follows that $\mu_R(\fm)=2$. Assume
$\fm=(x_1,x_2)$. We show that $x_1,x_2$ is a sequence of exact zero-divisors. But as $R$, $R/(x_1)$ and $R/(x_1,x_2)$
are complete intersection, it follows by Proposition \ref{T5} that $x_1,x_2$ is a sequence of exact zero-divisors.

Conversely if $R$ admits a minimal sequence of exact zero-divisors of length $2$, then by Corollary \ref{C21},
R is complete intersection and $\mu_R(\fm)=2$. It follows that $\H_R(t)=(1+t)^2$ and so by \cite[Corollary 4.5]{HS},
R is Koszul.

(ii) It follows from Corollary \ref{C21} and \cite[Corollary 1.12]{HS}.

(iii) Let $Q=k[Y_1,Y_2,Y_3]$ and $S=k[X_1,X_2,X_3]$, where $Y_i$ and $X_i$ are indeterminates, $1\leq i\leq 3$.
$S$ has a structure of $Q$--module by multiplication $Y_iG:=\partial G/\partial X_i$,
$1\leq i\leq 3$ and $G\in S$. As $R$ is complete intersection, by \cite[Theorem 2.1]{MW}, $R\cong Q/(0:_QF)$, for
some homogeneous polynomial $F\in S$ of degree 3. So we may assume that $R= Q/(0:_QF)$.  As $k$ is algebraically closed,
then $V(F)=\{v\in k^3|F(v)=0\}\subset k^3$ is infinite set. Choose $v \in V(F)$ such that $v\neq 0$. Assume $v=(a_1,a_2,a_3)\in k^3$ and set
$l=a_1X_1+a_2X_2+a_3X_3\in Q$. Let $R=\oplus^{3}_{i=0}R_i$ where $R_0=k$ and consider the $k$--module map
$h:R_1\overset{l}\longrightarrow R_2$. As $F(a_1,a_2,a_3)=0$, it follows from \cite[Theorem 3.1]{MW} that
$h$ is not injective. So there is a linear form $l'\in R_1$ such that $l'\neq 0$ and $ll'=0$. We show that $(l,l')$ is a pair of exact zero-divisor on $R$.
As $ll'=0$ we have $Rl\subseteq(0:_Rl)$. Since $R$ is a Koszul complete intersection so $R=Q/(\alpha_1,\alpha_2,\alpha_3)$
where each $\alpha_i$ is quadric and $\alpha_1,\alpha_2,\alpha_3$ is $Q$--regular sequence. As $ll'=0$ in $R$ so $ll'=\Sigma^3_{i=1}\lambda_i\alpha_i$ in $Q$,
where $\lambda_i\in k$, $1\leq i\leq 3$. It follows that $(l,\alpha_1,\alpha_2,\alpha_3)=(l,\alpha_1,\alpha_2)$ and so that $R/(l)$ is a complete intersection.
Hence $(0:_Rl)$ is principal ideal and so that $(0:_Rl)=Rl'$. Now considering $R/(l)$, as by \cite[Corollary 1.12]{HS}, $R/(l)$ is Koszul complete
intersection, it follows by part (i) that $R/(l)$ admit a minimal sequence of exact zero-divisors of length $2$. Hence $R$ admits a minimal sequence
of exact zero-divisors of length $3$.
\end{proof}


\section{strong sequence of exact zero-divisors}
We define strong sequences of exact zero--divisors and establish
some conditions, under which, a sequence of exact zero--divisors is
a strong one. We also study local rings whose maximal ideals are
generated by a strong sequence of exact zero divisors.
\begin{defn} \emph{Assume that $x_1, \cdots, x_n\in R$ is a
sequence of exact zero-divisors on $R$ with the twins $y_1, \cdots,
y_n$, respectively. We introduce the following terminologies.}
\begin{itemize}
\item[(a)]\emph{ $x_1,\cdots,x_n$ is a \emph{permutable} sequence of exact
zero-divisors on $R$ if every permutation of it is again a sequence
of exact zero-divisors on $R$.}
\item[(b)]\emph{ $x_1,\cdots,x_n$ is an \emph{strong} sequence of exact zero-divisors on $R$,
if for any choice of distinct elements $i_1, \cdots, i_k$ of  $\{1, \cdots,
n\}$, $k<n$, and for each $j\in\{1,\cdots,n\}\backslash\{i_1, \cdots,
i_k\}$, $(x_j,y_j)$ is a pair of exact zero--divisors on $R/(x_{i_1}, \cdots, x_{i_k})$.}
\end{itemize}
\end{defn}


It is clear that any strong sequence of exact zero-divisors is permutable and any permutable sequence of exact
zero-divisors is minimal. Indeed assume $x_1,\cdots,x_n$ is a permutable sequence of exact zero-divisors.
If it is not minimal then there is a $j$, say $j=1$, such that $x_1\in(x_2, \cdots, x_n)$.
As $x_2, \cdots, x_n, x_1$ is also a sequence of
exact zero-divisors on $R$, $x_1$ is an exact zero-divisors on $R/(x_2, \cdots, x_n)$. But $\overline{x_1}=0$ in
$R/(x_2, \cdots, x_n)$ which is a contradiction.


\begin{exam}\label{Ex2} \emph{Let $K$ be a field and let $R=K[X_1, X_2]/({X_1}^2+{X_2}^2, X_1X_2)$. Denote by $x_1, x_2$ the images of
$X_1,X_2$ in $R$, respectively. Then one can check that $x_1, x_2$ is
a permutable sequence of exact zero-divisors on $R$ but it is not
strong. Indeed $x_1,x_2$ is a sequence of exact zero-divisors on $R$ with twins $x_2,x_2$
and $x_2,x_1$ is a sequence of exact zero-divisors on $R$ with the twins $x_1,x_1$.
So $x_1,x_2$ is not a strong sequence of exact zero-divisors.}
\end{exam}
In this section we study minimal and strong sequences of
exact zero--divisors on a local ring. It is clear that if $x_1,\cdots,x_n$ is a minimal(permutable, strong) sequence of
 exact zero--divisors on
$R$, then $x_1,\cdots,x_i$ is also a minimal(permutable, strong) sequence of exact
zero--divisors on $R$ and $x_{i+1},\cdots, x_n$ also is a minimal(permutable, strong)
sequence of exact zero--divisors on $R/(x_1,\cdots,x_i)$ for all $i$, $1\leq i\leq n$.


\begin{prop}\label{C7}\ Let $R$ be a local ring and let
$x_1, \cdots, x_n$ be non-zero and non-unit elements of $R$. Then
the following statements are equivalent.
\begin{itemize}
\item[(i)] $x_1, \cdots, x_n$ is a strong sequence of exact zero--divisors on
$R$.
\item[(ii)] For any choice of distinct elements $i_1, \cdots,i_k$ of  $\{1, \cdots,
n\}$ and for each $j\in\{1,\cdots,n\}\backslash\{i_1,\cdots,i_k\}$,
$x_j$ is an exact zero--divisor on $R$ and
$$\emph\Ext^i_R(R/(x_j),R/(x_{i_1}, \cdots, x_{i_k}))=0$$ for all
$i>0$.
\item[(iii)] For any choice of distinct elements $i_1, \cdots, i_k$ of  $\{1, \cdots,
n\}$ and for each $j\in\{1,\cdots,n\}\backslash\{i_1, \cdots,
i_k\}$, $x_j$ is an exact zero--divisor on $R$ and $$\emph\Tor^R_i(R/(x_{i_1}, \cdots,
 x_{i_k}),R/(x_j))=0$$ for all $i>0$.
\end{itemize}
\end{prop}
\begin{proof} (i)$\Rightarrow$(ii) and (i)$\Rightarrow$(iii).
By assumption each $x_j$ is an exact zero--divisor on $R$.
For each $j$, denote $y_j$ as a corresponding twin of $x_j$. As by definition, $(x_j,y_j)$ also is a pair of
exact zero--divisors on $R/(x_{i_1},\cdots,x_{i_k})$ for any choice
of distinct elements $i_1, \cdots, i_k$ of $\{1, \cdots, n\}$ such
that $j\notin\{i_1, \cdots, i_k\}$, hence the results follows by
Lemma \ref{L3}.

(iii)$\Rightarrow$(i) and (ii)$\Rightarrow$(i). It follows from Lemma \ref{L3}.
\end{proof}

The following lemma is an easy exercise.


\begin{lem}\label{L7} Let $R$ be a ring, and let $M$, $N$, $K$ and
$L$ be $R$--modules. Consider a sequence
\begin{equation}\label{Eq1}M\overset{f}\longrightarrow K\oplus L\overset{g}\longrightarrow N
\end{equation} of $R$--homomorphisms such
that $\emph\Im f\subseteq K\times(0)\subseteq\emph\Ker g$. Then (\ref{Eq1}) is exact if and only if the sequences
$M\overset{f'}\longrightarrow K\longrightarrow 0$ and $0\longrightarrow
L\overset{g'}\longrightarrow N$ are exact, where $f'$ is the composition
$M\overset{f}\longrightarrow K\oplus L\overset{nat.}{\longrightarrow} K$ and  $g'$
is the composition $L\overset{nat.}{\longrightarrow} K\oplus L\overset{g}{\longrightarrow}N$.
\end{lem}


\begin{prop}\label{P8}\ Let $R$ be a local ring and $x_1,\cdots,x_n$
elements in $R$. Then $x_1,\cdots,x_n$ is a strong
sequence of exact zero-divisors if and only if it is minimal and there exist twins $y_1,\cdots,y_n$ of $x_1,\cdots,x_n$, respectively, such
that $x_jy_j=0$, $1\leq j\leq n$.
\end{prop}
\begin{proof} $(\Rightarrow)$ It is clear by definition.

$(\Leftarrow)$ That is enough to prove that $x_2,x_1,\cdots,x_n$ is a
sequence of exact zero-divisors with the twins $y_2,y_1,\cdots,y_n$, respectively. Hence we may assume that $n=2$.
Let $\K_{\bullet}(x_1)$, $\K_{\bullet}(x_2)$ and $\K_{\bullet}(x_1,x_2)$
be the Koszul complexes of $R$ with respect to the ideals $(x_1)$,
$(x_2)$ and $(x_1,x_2)$, respectively, and $\H_*(x_1)$, $\H_*(x_2)$
and $\H_*(x_1,x_2)$ be their corresponding homologies. We prove the
claim in two steps.

Step 1. We prove $(x_1,y_1)$ is a pair of exact zero--divisor on $R/(x_2)$. Consider the exact sequences of
complexes
$$0\longrightarrow
K_\bullet(x_1)\longrightarrow\K_\bullet(x_1,x_2)\longrightarrow
K_\bullet(x_1)(-1)\longrightarrow 0$$ and $$0\longrightarrow
K_\bullet(x_2)\longrightarrow\K_\bullet(x_1,x_2)\longrightarrow
K_\bullet(x_2)(-1)\longrightarrow 0$$ which imply the long exact
sequences of homologies
\begin{equation}\label{Eq6} 0\longrightarrow (0:_R(x_1,x_2))\longrightarrow
(0:_Rx_1)\overset{-x_2}\longrightarrow\end{equation}
$$(0:_Rx_1)\overset{f}{\longrightarrow} \H_1(x_1,x_2)
\overset{g}{\longrightarrow} R/(x_1)\overset{x_2}\longrightarrow
R/(x_1)\longrightarrow R/(x_1,x_2)\longrightarrow 0,$$ and
\begin{equation}\label{Eq3}0\longrightarrow
(0:_R(x_1,x_2))\longrightarrow
(0:_Rx_2)\overset{x_1}\longrightarrow\end{equation}
$$(0:_Rx_2)\overset{f'}{\longrightarrow}
\H_1(x_1,x_2)\overset{g'}{\longrightarrow}
R/(x_2)\overset{x_1}\longrightarrow R/(x_2)\longrightarrow
R/(x_1,x_2)\longrightarrow 0,$$  where $-x_2$ and $x_1$ are
connecting homomorphisms and $f, g, f'$ and $g'$ are the natural
maps. Denote the $i$th differential of $\K_{\bullet}(x_1,x_2)$ by
$d_i:\K_i(x_1,x_2)\longrightarrow\K_{i-1}(x_1,x_2)$. Then $\im
d_2=R\langle-x_2,x_1\rangle$ where $R\langle-x_2,x_1\rangle$ is
$R$--submodule of $\K_1(x_1,x_2)$ generated by the element
$\langle-x_2,x_1\rangle\in\K_1(x_1,x_2)$. From (\ref{Eq6}) we obtain
the following exact sequence
\begin{equation}\label{Eq2}0\longrightarrow
(0:_Rx_1)/x_2(0:_Rx_1)\overset{\overline{f}}\longrightarrow
\H_1(x_1,x_2)\overset{\overline{g}}\longrightarrow
0:_{R/(x_1)}x_2\longrightarrow 0,\end{equation} where
$\overline{f}(\lambda+x_2(0:_Rx_1))=\langle\lambda,0\rangle+R\langle-x_2,x_1\rangle$
and $\overline{g}(\langle
u,v\rangle+R\langle-x_2,x_1\rangle)=v+(x_1)$. Note that
$(0:_Rx_1)/x_2(0:_Rx_1)=(y_1)/(x_2y_1)\cong R/(x_1,x_2)$ under the
map $ry_1+(x_2y_1)\longrightarrow r+(x_1,x_2)$. Note that as $x_2$ is
an exact zero--divisor on $R/(x_1)$ so that $\im\overline{g}=0:_{R/(x_1)}x_2\cong
R/(x_1,x_2)$. Hence the exact sequence (\ref{Eq2}) splits and so we
find that $\H_1(x_1,x_2)=\im\overline{f}\oplus K$ for some submodule $K$ of $\H_1(x_1,y_1)$ such that $K\cong\im\overline{g}$.
As by assumption $x_2y_2=0$, one can easily check that $K$ is generated by
$\langle 0,y_2\rangle + R\langle-x_2,x_1\rangle \in \H_1(x_1,x_2)$.

Next consider the long exact sequence (\ref{Eq3}). One can check that
$f'(\alpha)=\langle
0,\alpha\rangle+R\langle-x_2,x_1\rangle\in\H_1(x_1,x_2)$ for all
$\alpha\in(0:_Rx_2)$ and $g'(\langle
u,v\rangle+R\langle-x_2,x_1\rangle)=u+(x_2)\in R/(x_2)$ for all
$\langle u,v\rangle+R\langle-x_2,x_1\rangle\in\H_1(x_1,x_2)$. Hence
$\im f^{'}\subseteq 0\times K\subseteq \Ker g^{'}$. It
follows from Lemma \ref{L7} that the exact sequence (\ref{Eq3})
decomposes into the following two exact sequences
\begin{equation}\label{e1}0\longrightarrow (0:_R(x_1,x_2))\longrightarrow
(0:_Rx_2)\overset{x_1}\longrightarrow(0:_Rx_2)\longrightarrow
K \longrightarrow 0\end{equation} and
\begin{equation}\label{e2}0\longrightarrow \im\overline{f}\longrightarrow
R/(x_2)\overset{x_1}\longrightarrow R/(x_2)\longrightarrow
R/(x_1,x_2)\longrightarrow 0.\end{equation} By (\ref{e2}), we have
$0:_{R/(x_2)}x_1\cong\im\overline{f}\cong R/(x_1,x_2)$. In other
words $x_1$ is an exact zero--divisor on $R/(x_2)$. On the other hand
$0:_{R/(x_2)}x_1$ is generated by $y_1+(x_2)\in R/(x_2)$ which
implies that $(x_1, y_1)$ is a pair of exact zero--divisors on
$R/(x_2)$.\\

Step 2. As $K\cong R/(x_1,x_2)$, the exact
sequence (\ref{e1}) shows that $(0:_Rx_2)/x_1(0:_Rx_2)$
is a cyclic $R$--module. Now, Nakayama Lemma implies that $0:_Rx_2$
is a principal ideal. Let $z\in R$ be such that $(0:_Rx_2)=Rz$. Then $y_2\in Rz$ and
there is an exact sequence
\begin{equation}\label{Eq4}R\overset{z}\longrightarrow
R\overset{x_2}\longrightarrow R\longrightarrow
R/(x_2)\longrightarrow 0.\end{equation} As concluded in Step 1,
$(x_1,y_1)$ is pair of exact zero--divisors on $R/(x_2)$. Therefore,
by Lemma \ref{L3}, $\Tor_1^R(R/(x_2), R/(x_1))= 0$. As a result,
tensoring (\ref{Eq4}) by $R/(x_1)$ implies the exact sequence
$R/(x_1)\overset{z}\longrightarrow
R/(x_1)\overset{x_2}\longrightarrow R/(x_1)\longrightarrow
R/(x_1,x_2)\longrightarrow 0$. As $(x_2,y_2)$ is pair of exact zero--divisor on
$R/(x_1)$, it follows that $(x_1,y_2)/(x_1)=(x_1, z)/(x_1)$ and as $y_2\in Rz$, we have $Ry_2=Rz$.

Now we show that $(0:_Ry_2)=Rx_2$. Note that as $x_1,y_2$ also is a sequence of exact zero-divisors with the twins $y_1,x_2$, by
the last part, it is enough to show $x_1,y_2$ satisfies in the assumptions of the proposition. As $x_2y_2=0$, it is enough
to show that $\mu_R(x_1,y_2)=2$. Clearly we have $y_2\notin Rx_1$. Assume $x_1\in Ry_2$.
Then $x_1=ry_2$ for some $r\in R$. Hence $\overline{r}\overline{y_2}=0$
in $R/(x_1)$. As $(\overline{x_2},\overline{y_2})$ is a pair of exact
zero--divisors on $R/(x_1)$, we have
$\overline{r}\in(\overline{x_2})$. Hence we can write
$r=t_1x_1+t_2x_2$ where $t_1,t_2\in R$. Thus $x_1=t_1x_1y_2$ and
therefore $x_1=0$, which is a contradiction. Therefore $x_1,y_2$ is a
minimal sequence of exact zero--divisors. Hence similarly $(0:_Ry_2)=Rx_2$.
\end{proof}


\begin{cor}\label{C13} Let $R$ be a local ring and $x_1,\cdots,x_n$ be a minimal sequence of
exact zero-divisors on $R$ with twins $y_1,\cdots,y_n$ respectively.
Then $x_1,\cdots,x_n$ is a strong sequence of exact zero-divisors if and only if
for each $i$, $1\leq i\leq n$, there exist elements $r_1,\cdots,r_{i-1}\in R$ such that $y_i+r_1x_1+\cdots+r_{i-1}x_{i-1}\in (0:_Rx_i)$.
\end{cor}
\begin{proof} Replacing $y_i$ with $y_i+r_1x_1+\cdots+r_{i-1}x_{i-1}$, $1\leq i\leq n$, the result follows by Proposition \ref{P8}.
\end{proof}


\begin{rem} \emph{Note that in Proposition \ref{P15} (i), we can find a strong
sequence of exact zero divisors of length $2$ such that generates $\fm$. Indeed let
$\fm=(x_1,x_2)$. As we saw $x_1,x_2$ is a sequence of exact zero-divisors. Let $y_1,y_2$ be twins
of $x_1,x_2$ respectively. Then one can see that over $R/(x_1)$, $y_2+(x_1)=x_2+(x_1)$. Hence we can assume $y_2=x_2$.
If there is an element $r\in R$ such that $x_2+rx_1\in (0:_Rx_2)$, then we are done by Corollary \ref{C13}. Assume
for all $r\in R$, $x_2+rx_1\notin(0:_R x_2)$. It follows that $x_1x_2=0$. Now set $u=x_1+x_2$ and note that $0:_R u=Rz$
for some $z\in \fm$.
Then $z=\lambda_1x_1+\lambda_2x_2$ for some non-zero units $\lambda_1,\lambda_2 \in R$. Replacing $u$ with $x_2$ we are done.}
\end{rem}


\begin{exam}\label{Ex1} \emph{Let $R=K[X_1, X_2, X_3]/({X_1}^2, {X_2}^2 +X_1X_3, {X_3}^2)$ where $K$ is a field.
Let $x_1, x_2, x_3$ be the images of $X_1, X_2, X_3$ in $R$. Then
one can easily check that $x_1, x_2, x_3$ is a sequence of exact
zero-divisors on $R$. We have $\mu_R(x_1, x_2, x_3)=3$ and $R/(x_1,
x_2, x_3)\cong K$ is a regular ring, but $x_1, x_2, x_3$ is not
permutable. Indeed if $x_2, x_1, x_3$ is also a sequence of exact
zero-divisors, then by Proposition \ref{T5}, $R/(x_2)$ must be a
complete intersection but one can easily check that it is not even a
Gorenstein ring.}
\end{exam}


\begin{lem}\label{C6} Let $R$ be a local ring and let
$x_1, \cdots, x_n$ be a strong sequence of exact zero--divisors on
$R$ with twins $y_1,\cdots,y_n$, respectively. Then $z_1, \cdots,
z_n$ is a strong sequence of exact zero--divisors on $R$ whenever
$z_i\in\{x_i, y_i\}$ for all $i$, $1\leq i\leq n$.
\end{lem}
\begin{proof} It is enough to show that $y_1,x_2,\cdots, x_n$
is also a strong sequence of exact zero--divisors on $R$. By
Proposition \ref{C7}, $\Tor^R_i(R/(x_1),R/(x_{i_1}, \cdots,
x_{i_k}))=0$ for all $i>0$ and any choice $\{x_{i_1}, \cdots,
x_{i_k}\}\subseteq\{x_2,\cdots,x_n\}$. As $R/(y_1)$ is a
syzygy of $R/(x_1)$, it follows that $\Tor^R_i(R/(y_1),R/(x_{i_1},
\cdots, x_{i_k}))=0$ for all $i>0$. Let $j\notin\{1,i_1,\cdots,i_k\}$.
Then by Proposition \ref{T4} we have
\[\begin{array}{rl}\
\Tor^{R}_i(R/(x_j),R/(y_1,x_{i_1},\cdots,x_{i_k})) &\cong\Tor^{R/(y_1)}_i(R/(y_1,x_j),R/(y_1,x_{i_1},\cdots,x_{i_k}))\\
&\cong\Tor^{R}_i(R/(y_1,x_j),R/(x_{i_1},\cdots,x_{i_k}))\\
&\cong\Tor^{R/(x_j)}_i(R/(y_1,x_j),R/(x_j,x_{i_1},\cdots,x_{i_k}))\\
&\cong\Tor^{R}_i(R/(y_1),R/(x_j,x_{i_1},\cdots,x_{i_k}))=0.
\end{array}\]
Now the result follows from Proposition \ref{C7}.
\end{proof}


\begin{lem}\label{L6} Let $R$ be a local ring and let
$x_1, \cdots, x_n$ be a strong sequence of exact zero--divisors on
$R$. Then $\overline{x_1}, \cdots, \overline{x_n}$ is a strong
sequence of exact zero--divisors on $R/\alpha R$ for each
$R$-regular element $\alpha\in R$.
\end{lem}
\begin{proof}\ Let $\alpha$ be an $R$--regular element and set $\overline{R}=R/(\alpha)$. Let $y_i$ be the corresponding
twin of $x_i$, $1\leq i\leq n$.
It follows from Proposition \ref{P10} that $\overline{x_1}, \cdots,
\overline{x_n}$ is a sequence of exact zero--divisors on
$\overline{R}$. As $x_1, \cdots, x_n$ is a strong sequence of exact
zero-divisors, $\alpha$ is $R/(x_{i_k},\cdots,x_{i_k})$--regular
element, for any choice $\{x_{i_1}, \cdots,
x_{i_k}\}\subset\{x_1,\cdots,x_n\}$. Let $1\leq j\leq n$ such that
$j\notin\{i_1,\cdots,i_k\}$. Then we have by \cite[Lemma2 page 140
]{M},
$\Tor^{\overline{R}}_i(\overline{R}/(\overline{x_j}),\overline{R}/(\overline{x_{i_1}},\cdots,\overline{x_{i_k}}))
\cong
\Tor^R_i(R/(x_j),\overline{R}/(\overline{x_{i_1}},\cdots,\overline{x_{i_k}}))$.
But as $(x_j,y_j)$ is a pair of exact zero-divisors on
$\overline{R}/(\overline{x_{i_1}},\cdots,\overline{x_{i_k}})$,
$\Tor^R_i(R/(x_j),\overline{R}/(\overline{x_{i_1}},\cdots,\overline{x_{i_k}}))=0$,
by Lemma \ref{L3}. Now the result follows by Proposition \ref{C7}.
\end{proof}


For a local ring $R$, the following result gives us an upper bound
for the length of a strong sequence of exact zero--divisors on $R$
in terms of $\e(R)$, the multiplicity of $R$.


\begin{thm}\label{T6} Let $(R,\fm)$ be a local ring. Let $\emph\e(R)$ be the multiplicity of $R$ with respect
to the maximal ideal $\fm$. Then
\begin{itemize}
\item[(i)] The length of any strong sequence of exact zero--divisors on $R$ is less than or equal to $\emph\Log_2(\emph\e(R))$.
\item[(ii)] Assume that $R$ is unmixed. If there is a strong sequence of exact zero--divisors
on $R$ of length equal to the integer part of $\emph\Log_2(\emph\e(R))$, then $R$ is a complete
intersection.
\item[(iii)] Assume that $R$ is unmixed. If there is a strong sequence of exact zero--divisors $x_1,\cdots,x_n$
on $R$ such that $n=\emph\Log_2(\emph\e(R))$, then $R$ is a Koszul complete intersection.
\end{itemize}
\end{thm}
\begin{proof} (i) Let $x_1, \cdots, x_n$ be a strong sequence of
exact zero--divisors on $R$ and let $y_i$ be the twin of $x_i$ for
$i=1, \cdots, n$. For each $i$, $1\leq i\leq n$, Lemma \ref{C6}
implies that there exists an exact sequence
$$0\longrightarrow R/(z_1, \cdots, z_{i-1},x_i)\longrightarrow
R/(z_1, \cdots, z_{i-1})\longrightarrow R/(z_1, \cdots,
z_{i-1},y_i)\longrightarrow 0,$$ where $z_j\in\{x_j, y_j\}$ for $j=1,
\cdots, n$. By \cite[Theorem 14.6]{M}, we have
$$\e(R/(z_1, \cdots, z_{i-1}))=\e(R/(z_1, \cdots,
z_{i-1},x_i))+\e(R/(z_1, \cdots, z_{i-1},y_i)).$$  Therefore
\begin{equation}\label{D7}\e(R)=\underset{\underset{i=1, \cdots, n}{z_i\in\{x_i, y_i\}}}{\sum}\e(R/(z_1,
\cdots, z_n)).\end{equation} As $\dim R/(z_1, \cdots, z_n)=\dim R$, one has $\e(R/(z_1, \cdots, z_n))>0$,
and so $\e(R)\geq 2^n$. Thus $n\leq\emph\Log_2(\emph\e(R))$.

(ii) Assume that there is a strong sequence of exact zero--divisors
on $R$ of length $n$ equal to the integer part of
$\Log_2(\e(R))$. By equation (\ref{D7}), there is an
strong sequence $x_1, \cdots, x_n$ of exact zero divisors on $R$
with $\e(R/(x_1, \cdots, x_n))=1$. As $\Ass_R R/(x_1, \cdots,
x_n)\subseteq\Ass R$, the ring $R/(x_1, \cdots, x_n)$ is unmixed
too. Hence, by \cite[Theorem 40.6]{N}, $R/(x_1, \cdots, x_n)$ is a
regular ring. As $(x_1, \cdots, x_n)$ is a quasi-complete intersection ideal, by \cite[7.5]{AHS}, $R$ is complete intersection.

(iii) By last part $R$ is complete intersection and $R/(x_1,\cdots,x_n)$ is regular ring. Let $d=\dim R$. We can choose a regular
sequence $u_1,\cdots,u_d$ in $R$ such that their image in $R/(x_1,\cdots,x_n)$ is a regular system of parameters.
By Lemma \ref{L6}, $x_1,\cdots,x_n$ also is a strong sequence of exact zero-divisors on $R/(u_1,\cdots,u_d)$. Hence we may assume that
$d=0$. Thus $\fm=(x_1,\cdots,x_n)$ and as $\e(R)=2^n$, the result follows by \cite[4.5]{HS}.
\end{proof}


As we saw in Example \ref{Ex1}, a sequence of exact zero-divisors
which the quotient ring is regular may not be even permutable. In
the following we give condition on a complete intersection ring to
have a strong sequence of exact zero-divisors.


\begin{prop}\label{T11} Let $(S,\fn)$ be a regular local ring and let
$I\subseteq \fn^2$ be an ideal of $S$ which is generated by an
$S$--regular sequence $\alpha_1, \cdots, \alpha_t$. Assume that for
$i=1, \cdots, t$, each ideal $(\alpha_i)$ has an associated prime
ideal which is not contained in $\fn^2$. Set $R=S/I$ and
$\fm=\fn/I$. Then $R$ has a strong sequence of exact zero--divisors,
of maximal length $\emph\dim S-\emph\dim R$, contained in
$\fm\backslash\fm^2$. Moreover if $x_1, \cdots, x_t$ is such a
sequence, then $R/(x_1, \cdots, x_t)$ is a regular ring.
\end{prop}
\begin{proof} First we prove the case $I= (\alpha)$ where $\alpha\in\fn^2$ be a non--zero element. Let $S\alpha=\cap^s_{i=1}\fq_i$
 be the minimal
primary decomposition of the ideal $I$ where $\fq_i$ is
$\fp_i$--primary ideal for all $i=1,\cdots,s$. We have $\h \fp_i=1$
and as $S$ is UFD, $\fp_i$ is principal ideal for all $i=1,\cdots,
s$ (see \cite[Theorem 20.1 and Theorem 20.3]{M}). For each $i$,
$i=1,\cdots, s$, we set $\fp_i=Sz_i$, and one can easily check that
$\fq_i=Sz^{k_i}_i$ for some positive integer $k_i$ and so
$S\alpha=Sz^{k_1}_1\cdots z^{k_s}_s$. By our assumption, we may
assume that $z_1\in\fn\backslash\fn^2$. Set $x$ and $y$ for the
images of $z_1$ and $z^{k_1-1}_1\cdots z^{k_s}_s$ in $R$,
respectively. Then it follows from \cite[Example 2.4]{H} that
$(x,y)$ is a pair of exact zero--divisors on $R$. Note that
$R/(x)=S/(z_1)$ is a regular ring.

Next let $t> 1$. Set $S':=S/(\alpha_1)$ and denote $(-)': S\longrightarrow S'$ the natural map. As mentioned, $S'$ has a pair of
 exact zero--divisors $(x^{\prime}_1,y^{\prime}_1)$, such that
$x^{\prime}_1\in \fn^{\prime}\backslash {\fn^{\prime}}^2$, where $\fn^{\prime}=\fn/(\alpha_1)$ is
the maximal ideal of $S^{\prime}$. Note that  $\alpha^{\prime}_2,\cdots,\alpha^{\prime}_t$
is an $S^{\prime}$--regular sequence so that, by Proposition \ref{P10}, $(x^{\prime}_1,y^{\prime}_1)$ is a pair of exact
zero--divisors on
$S^{\prime}/(\alpha^{\prime}_2,\cdots,\alpha^{\prime}_n)S^{\prime}=R$. Let $x_1$ and $y_1$
be the images of $x^{\prime}_1$ and $y^{\prime}_1$ in $R$, respectively. Then $(x_1,y_1)$ is a pair of exact
zero--divisors on $R$. As in the case $n=1$, choose an element $z_1\in S$ such that its
image in $S^{\prime}$ is $x^{\prime}_1$. Then we have $$R/(x_1)=
S^{\prime}/(x^{\prime}_1,\alpha^{\prime}_2,\cdots,\alpha^{\prime}_t)
=S/(z_1,\alpha_2,\cdots,\alpha_t)
=\frac{S}{(z_1)}\big{/}(\alpha_1,\cdots,\alpha_n)\frac{S}{(z_1)}.$$
Note that as $\alpha_1,\cdots,\alpha_t$ is a regular sequence in $S$, then the principal ideals $(\alpha_1),\cdots,(\alpha_t)$ have no
common prime divisors. For each $i$, $i=2, \cdots, t$, assume that $(u_i)$ denotes a prime devisor of $(\alpha_i)$ which is not
contained
in $\fn^2$.
Set $\overline{S}=S/(z_1)$. As $\overline{S}=S^{\prime}/(x^{\prime}_1)$ is a regular ring, one can check that the principal ideal
$(\overline{\alpha_i})$ has a prime divisor $(\overline{u_i})$ which
is not contained in $\overline{\fn}^2$. Note that $\overline{\alpha_2}, \cdots, \overline{\alpha_t}$ is
$\overline{S}$--regular sequence. Hence
the result follows by induction.
\end{proof}


$\mathbf{Acknowledgement}$ {The authors are grateful to Roger
Wiegand, Luchzar Avramov and Alexandra Seceleanu for their
 discussions and comments.}

\bibliographystyle{amsplain}

\end{document}